\documentclass[12pt]{amsart}
\usepackage{tabularx} %
\usepackage{fullpage,graphicx,color}

\theoremstyle{plain}
\newtheorem{theorem}{Theorem}

\newtheorem{corollary}[theorem]{Corollary}

\theoremstyle{definition}

\newtheorem{example}[theorem]{Example}

\newtheorem{problem}[theorem]{Problem}

\begin{document}

\title{Optimized random chemistry}

\date{\today}

\author{Jeffrey S. Buzas}
\address{Dept. of Mathematics \& Statistics\\
  University of Vermont \\
  Burlington, VT 05401}
\email{jbuzas@uvm.edu}

\author{Gregory S. Warrington}
\address{Dept. of Mathematics and Statistics\\
  University of Vermont \\
  Burlington, VT 05401}
\email{gregory.warrington@uvm.edu}

\thanks{Second author partially supported by National Science
  Foundation grant DMS-1201312}

\begin{abstract}
  The random chemistry algorithm of Kauffman can be used to determine
  an unknown subset $S$ of a fixed set $V$.  The algorithm proceeds by
  zeroing in on $S$ through a succession of nested subsets $V_0 = V
  \supset V_1 \supset \cdots \supset S$.  In Kauffman's original
  algorithm, the size of each $V_i$ is chosen to be half the size of
  $V_{i-1}$.  In this paper we determine the optimal sequence of sizes
  so as to minimize the expected run time of the algorithm.
\end{abstract}


\maketitle

\section{Introduction}

Consider the following set-guessing game between a responder and a
questioner.  The two players first agree on integers $n \geq k \geq
0$.  The game begins with the responder secretly choosing a subset $S
\subset [n] = \{1,2,\ldots,n\}$ of cardinality $k$.  The questioner's
task is to determine $S$.  During each turn, the questioner proposes a
set $V$; the responder then indicates whether or not $V$ contains $S$.
The game ends when the questioner proposes $V=S$.

In~\cite{kauffman1996home}, Kauffman proposes a \emph{Random
  Chemistry} (RC) algorithm for determining the set $S$ when $k \ll
n$.  In following this algorithm, the questioner zeroes in on $S$ by
choosing successively smaller sets.  More precisely, she creates a
sequence of sets $[n] = V_0 \supset V_1 \supset V_2 \supset \cdots
\supset V_M = S$ as follows.  Given a set $V_{i-1}$ known to strictly
contain $S$, the questioner proposes random subsets of $V_{i-1}$ of
cardinality $|V_{i-1}|/2$ until she finds one containing $S$.  This
set is then taken to be $V_i$.

By allowing the ratios $|V_i|/|V_{i-1}|$ to deviate from a ratio of
$1/2$, we obtain a more general class of algorithms.  In
Theorem~\ref{thm:main} we determine the sequence of ratios that
minimizes the expected number of turns in the game.

The above set-guessing problem has appeared in disparate applied
contexts.  In fact, Kauffman developed his RC algorithm as a way to
hypothetically search for auto-catalytic sets of molecules.  Eppstein
et al.~\cite{epp} later implemented an RC algorithm as a way of
searching for small sets of non-linearly interacting genetic
variations in genome-wide association studies.  More
recently, Eppstein and Hines~\cite{EppsteinHines} applied a slight
variation of this algorithm to search for collections of multiple
outages leading to cascading power failures in models of electrical
distribution grids.

A number of closely related problems have also been widely studied.
Most notably, the field of group testing (see, for
example,~\cite{grouptesting} for an overview) is also concerned with
finding unknown sets.
In group testing, a positive test result occurs when the pooled group
$V$ has a nonempty intersection with the unknown set $S$.  In the
problem we discuss, a positive test result occurs only when $V$
contains all of $S$.

Another class of related problems is that of searching games in which
the responder can lie.  The R\'enyi-Ulam game~\cite{renyi,ulam} is
probably the most famous of these; see~\cite{searchingerrors} for
a survey of such games.

In Section~\ref{sec:prob} we state precisely the optimization problem
we are trying to solve.  In Section~\ref{sec:sol} we solve the
continuous analog of the problem while Section~\ref{sec:numeric}
presents numeric data regarding how well the continuous solution
mirrors the discrete one.  Section~\ref{sec:calcofvar} presents an
approximate solution to the problem of Section~\ref{sec:prob} as an
application of the calculus of variations.

\section{The Optimization Problem}
\label{sec:prob}

Assume $k$, $n$ and $S$ are as given as in the Introduction; set
$n_0=n$ and $V_0 = S$.  If we choose a proper subset of $V_0$ of size
$n_1$ such that $n_0 > n_1\geq k$, then
\begin{equation}\label{eq:ratio}
  p_1=\frac{\binom{n_0-k}{n_1-k}}{\binom{n_0}{n_1}}
\end{equation}
is the probability of obtaining a subset containing $S$.  The expected
number of times we would have to select a subset of size $n_1$ until
we find one containing $S$ is therefore $1/p_1$.

Now consider a sequence $n_1,n_2,\dots,n_M$ where $n_0>n_1>n_2>\cdots
> n_M=k$.  Our \emph{Generalized Random Chemistry (GRC) algorithm}
begins by selecting sets of size $n_1$ until we find one, $V_1$, such
that $V_1\supset S$.  Such a $V$ must exist since, by hypothesis, $V_0
\supset S$ and $n_1\geq k$.  We then select subsets of size $n_2$ from
$V_1$ until we find a set $V_2$ containing $S$.  The process continues
until we have chosen $V_M$.  Define $p_i$ as the probability of
selecting a set $V_i$ of size $n_i$ containing $S$ from a set
$V_{i-1}$ of size $n_{i-1}$ known to contain $S$.  Then, as
in~(\ref{eq:ratio}),
\begin{equation}
  p_i=\frac{ \binom{n_{i-1}-k}{n_i-k}}{\binom{n_{i-1}}{n_i}}.
\end{equation}

Let the random variable $X_i$ represent the number of selections
needed to find a set $V_i$ containing $S$ as described above. Then the
expected value of $X_i$ is $1/p_i$ ($X_i$ is a geometric random
variable).  Let $X=X_1+\cdots +X_M$ denote the random variable
representing the total number of selections until we find $S$.  This
presents the following\\

\begin{problem}\label{prob}
  How should one choose $M\leq n_0-k$ and $n_1,n_2,\dots,n_M$ subject
  to $n_0>n_1>n_2>\dots > n_M=k$ so as to minimize $E[X]$?
\end{problem}

\section{The Continuous Solution}
\label{sec:sol}

In the combinatorial formulation leading to Problem~\ref{prob}, the
$n_i$ are all integers.  In this section we relax this requirement and
provide an optimal solution when the $n_i$ are only required to be
real.  Accordingly, we replace the factorial functions with
$\Gamma$-functions; recall that $\Gamma(n)=(n-1)!$ for $n$ a positive
integer.

\begin{theorem}\label{thm:main}
  The expected number of steps in the GRC algorithm
  \begin{equation}\label{eq:min}
    E[X]=\sum_{i=1}^M\frac{1}{p_i} =
    \sum_{i=1}^M\frac{\Gamma(n_{i-1}+1)\Gamma(n_i-k+1)}
    {\Gamma(n_i+1)\Gamma(n_{i-1}-k+1)}
  \end{equation}
  is minimized over the real numbers for $p_i = \binom{n_0}{k}^{-1/M}$.  The
  optimal value of $M$ is $\ln\binom{n_0}{k}$, in which case the
  expression for the $p_i$ reduces to $p_i = e^{-1}$ and $E[X] =
  e\ln\binom{n_0}{k}$.
\end{theorem}

\begin{proof}
  Let $z_i=1/p_i$ and note that $\prod_{i=1}^M z_i=\binom{n_0}{k}$.
  Then the problem reduces to finding $M$ and $z_1,\dots, z_M$ that
  minimize $\sum_{i=1}^Mz_i$ subject to $\sum_{i=1}^M \ln(z_i)=C$
  where $C=\ln \binom{n_0}{k} $ and $z_i \ge 1$ for $1\leq i\leq M$.

  The method of Lagrange multipliers instructs us to minimize
  \begin{equation}
    F(z_1,z_2,\dots,z_M,\lambda) =
    \sum_{i=1}^Mz_i-\lambda\left [\sum_{i=1}^M \ln(z_i)-C\right ]
  \end{equation}
  where $\lambda$ is the Lagrange multiplier.  Differentiating with
  respect to $z_1,\dots,z_M$ and $\lambda$, setting the derivatives
  equal to zero and solving gives the solution $z_i=\hat{z}_i =
  \binom{n_0}{k}^{1/M}$.  To find the optimal value for $M$, note that
  $\sum_{i=1}^M \hat{z}_i=Me^{C/M}$.  This function is minimized when
  $M = \hat{M} = C = \ln \binom{n_0}{k}$.  The optimal values for
  the probabilities are then $\hat{p}_i=e^{-C/\hat{M}}=e^{-1}$ and the
  expected number of steps using the $\hat{p}_i$ is then
  $E[X]=\sum_{i=1}^{\hat{M}} 1/ \hat{p}_i=e\ln\binom{n_0}{k}$.
\end{proof}

While Theorem~\ref{thm:main} gives closed forms for the optimal values
of $M$ and the $p_i$, we do not obtain a simple expression for the
$n_i$.  For fixed $M$, the optimal sequence $n_1,n_2,\dots,n_{M-1}$
can be obtained by successively solving the equations
\begin{equation}\label{eq:nieqn}
  \frac{\binom{n_0}{n_1}}{\binom{n_0-k}{n_1-k}} =
  \binom{n_0}{k}^{1/M}.
\end{equation}
The equations are easily solved using a univariate root finding
algorithm.

Alternatively, we can modify equation~(\ref{eq:nieqn}) by using the
approximation $\binom{a}{b}\approx (a/b)^b$.  (This approximation
works best for $b\ll a$.)  In doing so, we find that the optimal
values of the $n_i$ are
\begin{equation}\label{eq:niapprox}
  n_i \approx k^{i/M}n_0^{1-i/M}
\end{equation}
and that the optimal value of $M$ is approximately $k\ln
\left(\frac{n_0}{k}\right)$.  As shown in Section~\ref{sec:calcofvar},
this approximate solution can be obtained directly by applying the
binomial coefficient approximation to equation~(\ref{eq:min}) and then
applying the calculus of variations.

From Theorem~\ref{thm:main}, the optimal solution has the property
that the $p_i$ are constant. The following corollary follows from the
well-known fact that a sum of independent, identically distributed
geometric random variables has the negative binomial distribution.

\begin{corollary}\label{cor:main}
  Fix $n_0,k$ and $M$. Then the sequence of $n_1,n_2,\dots,n_M$
  minimizing $E[X]$ induces the distributional property that $X\sim $
  Negative Binomial$(M,p)$ where $p= \binom{n_0}{k}^{-1/M}$.  That is,
  $X$ has probability mass function
  $P(X=x)=\binom{x-1}{M-1}(1-p)^{x-M}p^M$ for $x=M,M+1,\dots$.
\end{corollary}
  
The utility of Corollary~\ref{cor:main} is that it provides a
straightforward means for computing the probability that more than $l$
steps would be required to find the solution using the optimal sequence.

\section{Numerical data}
\label{sec:numeric}

Equation~(\ref{eq:niapprox}) gives a closed form approximate solution
to the continuous problem motivated by Problem~\ref{prob}.  To utilize
this solution in actual problems, we need to map the $n_i$ to
integers.  A simple method is to provisionally set each according
to~(\ref{eq:niapprox}).  Then, for $i=M-1,M-2,\ldots,2,1$ set $n_i =
\max(n_{i+1}+1,\lfloor n_i\rfloor)$.  In this section we will explore
several examples in order to show that the integer-valued approximate
solutions compare favorably with the continuous minimum
of~(\ref{eq:min}).

\begin{example} 
  Let $n_0=100$ and $k=5$.  The optimal number of partitions given by
  Theorem~\ref{thm:main} is $ M_{\mbox{opt}} = \ln\binom{n_0}{k}
  \approx 18$.  Using this optimal value for $M$, we computed values
  for $n_1,\dots,n_{M-1}$ using equations~(\ref{eq:nieqn})
  and~(\ref{eq:niapprox}).  Figure~\ref{fig:staircase} illustrates the
  very similar sequences that result.  The expected number of steps
  are $49.3$ for the exact solution and $50.5$ for the approximate
  solution.
\end{example}

\begin{center}
  \begin{figure}
\includegraphics[scale=0.40]{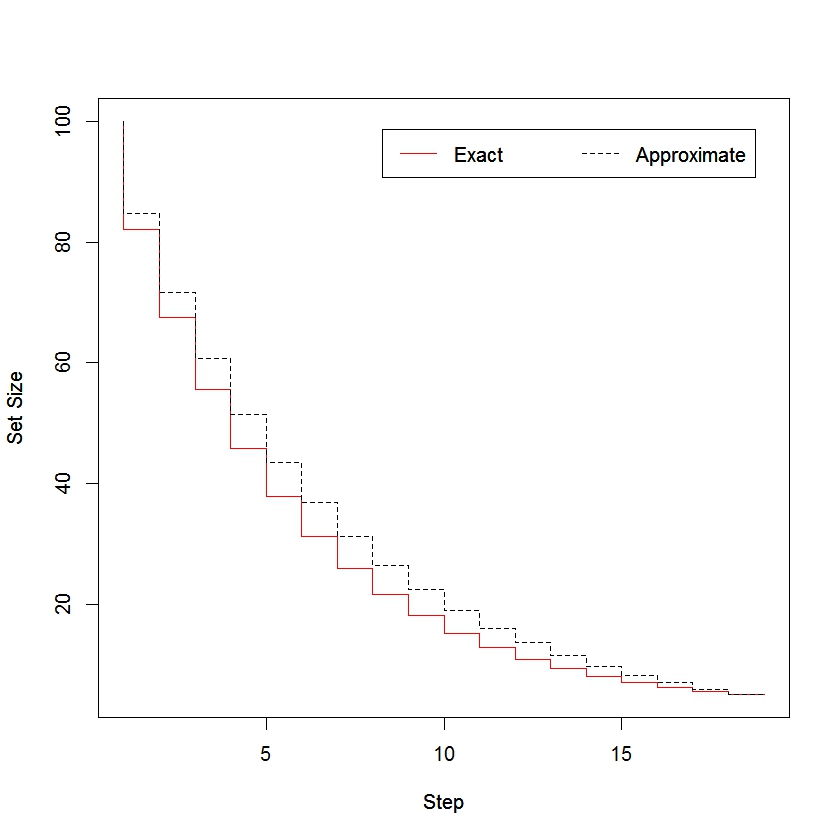} 
\caption{}
\label{fig:staircase}
\end{figure}
\end{center}

We also simulated $100,000$ runs of the GRC algorithm for the same
values of $n_0=100$ and $k=5$.  In each experiment, we generated a set
$S$ and counted the number of steps to find it using the optimal
sequence on the integer scale.  The mean number of steps was $50.9$.
Figure~\ref{fig:comparison} shows the empirical distribution for the
number of steps with the negative binomial distribution overlaid.

\begin{center}
\begin{figure}
\includegraphics[scale=0.40]{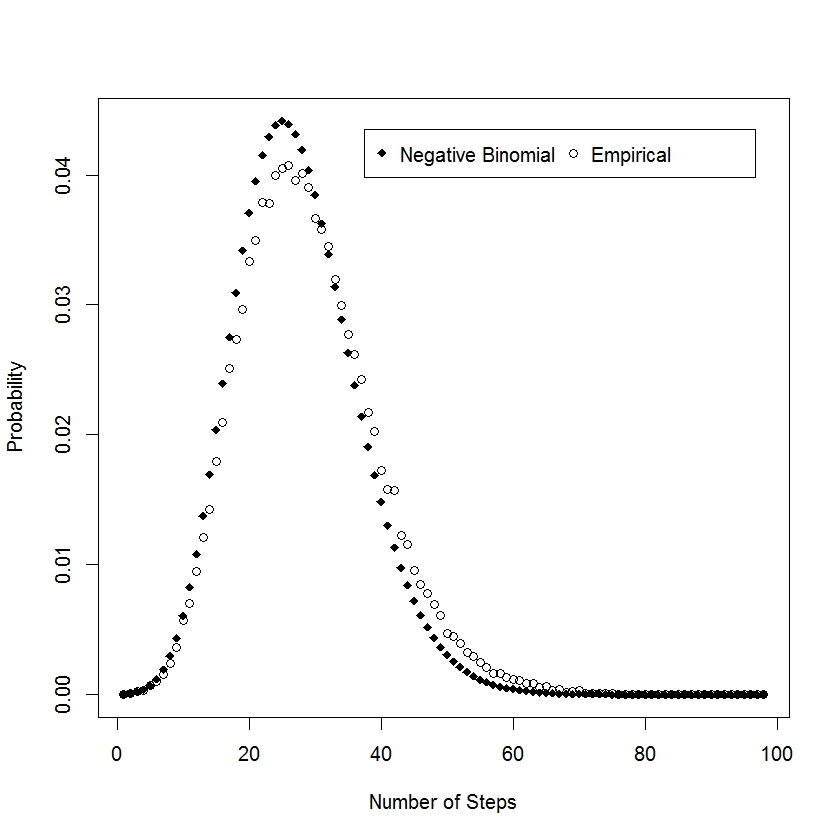} 
\caption{Comparison of Negative Binomial distribution to empirical
  distribution for number of steps based on 100,000 simulations where
  $n_0=100$ and $k=5$ }
\label{fig:comparison}
\end{figure}
\end{center}

Finally, in Figure~\ref{fig:curves} we compare the optimal solution of
Theorem~\ref{thm:main} to that of Kauffman's original RC algorithm.
Three different $n_i$-sequences are illustrated for $n_0=100$, $k=5$:
The exact solution corresponding to the solution of
equation~\ref{eq:nieqn} (circles), the approximate solution
of~\ref{eq:niapprox} (squares), and the sequence generated by
Kauffman's RC algorithm (triangles).  The paired increasing plots
illustrate the cumulative expected number of guesses required to find
a given set $V_i$.

\begin{center}
\begin{figure}
\includegraphics[scale=.5]{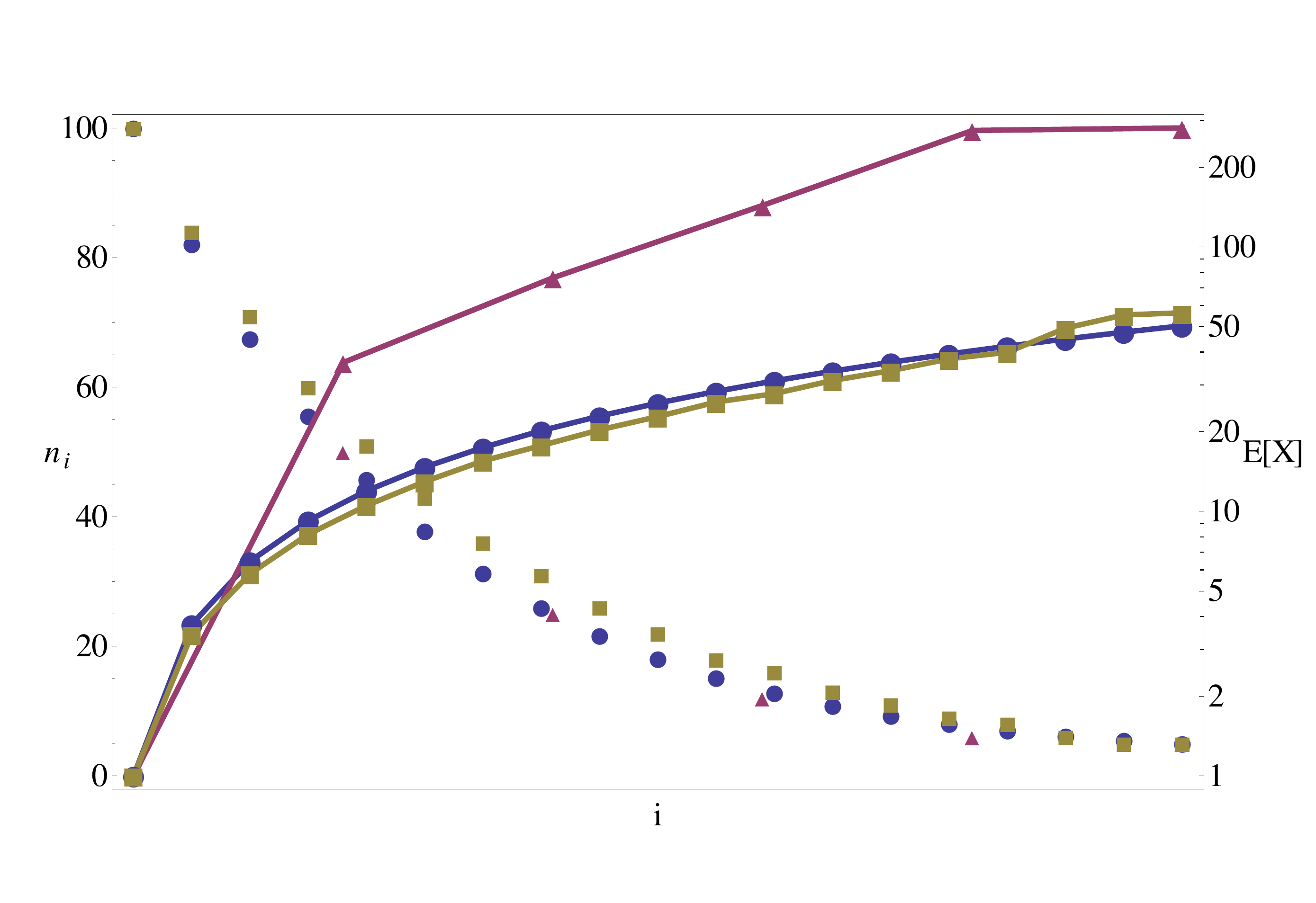} 
\caption{Comparison of various solutions to the subset-guessing
  problem: equation~(\ref{eq:nieqn}) (circles), equation~(\ref{eq:niapprox})
  (squares), and Kauffman's original RC algorithm (triangles).  The
  joined plots give the corresponding expected run times on a logarithmic $y$-scale.}
\label{fig:curves}
\end{figure}
\end{center}

\section{Calculus of Variations}
\label{sec:calcofvar}

In this section we present an alternate derivation of the
approximately optimal formula for the $n_i$ given
in~(\ref{eq:niapprox}).  In passing from the discrete to the
continuous, we can hope that an optimal solution $y(x)$ to
\begin{equation*}
  \int_0^M
  \frac{\Gamma((y(x)-1)+1)\Gamma(y(x)-k+1)}
  {\Gamma(y(x)+1)\Gamma((y(x)-1)-k+1)}\,dx
\end{equation*}
yields a good solution to the original sum.

The calculus of variations is designed for exactly this sort of
problem: It can be used to find functions $y(x)$ that are extrema of
the integral
\begin{equation*}
  \int_a^b f(y(x),y'(x); x)\,dx.
\end{equation*}
In this case we would set $n_i = y(i)$.  Unfortunately, it is not
clear that the required computation is tractable.  So we use the
aforementioned approximation
\begin{equation}\label{eq:covapprox}
  \frac{\Gamma(n_{i-1}+1)\Gamma(n_i-k+1)}{\Gamma(n_i+1)\Gamma(n_{i-1}-k+1)} 
  \approx \left(\frac{n_{i-1}}{n_i}\right)^k.
\end{equation}
In order to write the righthand side of~(\ref{eq:covapprox}) in
the form $f(y(x),y'(x); x)$, we use the approximation $y'(i) \approx
y(i)-y(i-1)$.  Simple algebra then shows that $y(i-1)/y(i) \approx 1 -
y'(i)/y(i)$.  Hence we can set $f(y(x),y'(x);x) =
\left(1-y'(x)/y(x)\right)^k$.  Euler's equation then tells us that we
should look for solutions to the differential equation
\begin{equation*}
  \frac{\partial f}{\partial y} = \frac{d}{dx}\frac{\partial f}{\partial y'}.
\end{equation*}
Straightforward computations yield
\begin{align*}
  \frac{\partial f}{\partial y} &= k\left(1-\frac{y'}{y}\right)^{k-1}\frac{y'}{y^2}\\
  \frac{\partial f}{\partial y'} &=
  k\left(1-\frac{y'}{y}\right)^{k-1}\left(\frac{-1}{y}\right).\\
  \frac{d}{dx}\frac{\partial f}{\partial y'} &=
  k(k-1)\left(1-\frac{y'}{y}\right)^{k-2} \frac{d}{dx}\left[1-\frac{y'}{y}\right]\left(\frac{-1}{y}\right) +
  k\left(1-\frac{y'}{y}\right)^{k-1}\frac{y'}{y^2}.
\end{align*}
One family of solutions will be those satisfying $y'=y$, i.e., $y(x) =
Ce^x$.  However, these are increasing functions so we ignore them.
Other solutions are those satisfying
\begin{equation*}
k\left(1-\frac{y'}{y}\right)\frac{y'}{y^2} =
k(k-1) \frac{d}{dx}\left[1-\frac{y'}{y}\right]\left(\frac{-1}{y}\right) +
  k\left(1-\frac{y'}{y}\right)\frac{y'}{y^2}.
\end{equation*}
This is satisfied by those $y$ for which $\frac{d}{dx}(1-y'/y)=0$ or,
equivalently, those for which $y'/y$ is a constant.  This implies
$y(x) = C_1e^{C_2 x}$.\\

The constraints $y(0) = n_0$ and $y(M) = n_M = k$ imply 
that our approximate solution is
\begin{equation*}
  y(x) = n_0e^{\frac{1}{M}\ln\left(\frac{k}{n_0}\right)x} =
  n_0\left(\frac{k}{n_0}\right)^{x/M} = n_0^{1-x/M}k^{x/M}.
\end{equation*}
By setting $x = i \in \{0,1,\ldots,M\}$, it follows that we can set
$\displaystyle{n_i = k^{i/M}n_0^{1-i/M}}$, as seen in
equation~(\ref{eq:niapprox}).

\section{Acknowledgments}

We wish to thank Jeff Dinitz, Margaret J. Eppstein, Stuart Kauffman
and Mark Wagy for helpful discussions leading to this paper.

\bibliographystyle{plain}
\bibliography{set_search}

\end{document}